\documentclass[12pt]{amsart}
\usepackage{amssymb,latexsym}

\makeatletter
\@namedef{subjclassname@2010}{%
  \textup{2010} Mathematics Subject Classification}
\makeatother

\newtheorem{thm}{Theorem}[section]
\newtheorem{cor}[thm]{Corollary}
\newtheorem{lem}[thm]{Lemma}

\theoremstyle{definition}

\numberwithin{equation}{section}

\def\sfrac#1#2{#1/#2}

\begin{document}


\baselineskip=17pt



\title[Roots of Markoff quadratic forms]{Roots of Markoff quadratic forms as strongly badly approximable numbers}

\author[J. Florek]{Jan~Florek}
\address{Institute of Mathematics, \\
University of Economics,
ul. Komandorska 118/120\\
  53--345 Wroc{\l}aw, Poland}
\email{jan.florek@ue.wroc.pl}
\date{}


\begin{abstract}
For a real number $x$, $\| x\| = \min \{|x-p|: p\in Z\}$ is the distance of  $x$ to the nearest integer. We say that two real numbers $\theta$, $\theta'$ are $\pm$~equivalent 
if their sum or difference is an integer. Let $\theta$ be irrational and put
\[
\varphi(\theta) = \inf \{ q \,\| q \theta\| : q \in N \}.
\]
We will prove: If  $\varphi(\theta)> 1/3$, then $\theta$
is $\pm$~equivalent to a root of $f_m (x,1) = 0$, where $f_m$ is a Markoff form. Conversely, if $\theta$ is $\pm$~equivalent to a root of
$f_m(x,1)=0$, then
\[
\varphi(\theta) = m \| m\theta \| = \frac{2}{3+\sqrt{9-4m^{-2}}} > 1/3.
\]
\end{abstract}



\maketitle


\section{Introduction}

By a \textsl{Markoff triple} we mean a solution $(m,m_1,m_2)$ of the following diophantine equation
\begin{equation}\label{fleq1.1}
m^2 + m^2_1 +m^2_2 = 3m m_1 m_2, \quad (m,m_1,m_2 \in Z, \ m \geq m_1 \geq m_2 \geq 1).
\end{equation}
A \textsl{Markoff number} is a member of such triple. 
The three integers in any Markoff triple are relatively coprime in pairs (see Cassels \cite{flref4} or Cusick--Flahive \cite{flref5}). A \textsl{Markoff form} associated to a~Markoff triple $(m,m_1,m_2)$ is defined as an indefinite binary  quadratic form with integral coefficients as follows. Define $u$ to be the least positive solution of
$\pm m_2 x \equiv m_1 \mod (m)$ (this is sensible since $m$ and $m_2$ are relative prime) and define $v = (u^2 + 1)/ m$ (note $m$ divides $u^2 +1$ because of (\ref{fleq1.1})). The Markoff form associated to a~Markoff triple $(m,m_1,m_2)$ is then defined as
\[
f_m(x,y) = mx^2 + (3m-2u)xy + (v-3u)y^2\, .
\]

Let $f(x,y) = \alpha{x^2} + \beta{xy} +\gamma{y^2}$ be a real indefinite binary quadratic form. The discriminant $\delta(f) =  \beta^2 - 4\alpha\gamma$ is then strictly positive. We write
\[
\mu(f) = \hbox{inf} \, |f(x,y)| \quad \mbox{($x$, $y\/$ integers not both 0)}.
\]
Two quadratic forms $f(x,y)$, $f'(x,y)$ are \textsl{equivalent} if there are integers $a$,~$b$, $c$, $d$ such that $ad-bc = \pm 1$ and $f'(ax + by, cx +dy) = f(x,y)$, identically in $x$, $y$. Markoff \cite{flref11}, \cite{flref12} (see also Cassels \cite{flref4}, Dickson \cite{flref6}, Frobenius  \cite{flref7}, Korkine--Zolotareff \cite{flref10} and Remak \cite{flref14}) showed that for any real indefinite binary quadratic form $f$, the inequality $\mu(f) / \delta(f) > 1/3$ holds if and only if $f$ is equivalent to a multiple of a Markoff form.

For a real number $x$, $\| x\| = \min \{|x-p|: p\in Z\}$ is the distance of  $x$ to the nearest integer. Let $\theta$ be irrational and let us denote
\[
\begin{array}{lll}
v(\theta) = & \lim \inf \{ q \,\| q \theta\| : & q \in N \},
\cr
\varphi(\theta) = & \inf \{ q \,\| q \theta\| : &q \in N \}.
\end{array}
\]
Two real numbers $\theta$ and $\theta' = \frac{a \theta + b}{c\theta +d}$, where $a,b,c,d$ are integers with $ad-bc = \pm1$, are called \textsl{equivalent}. If $\theta$ is equivalent to $\theta'$, then $v(\theta) = v(\theta')$  (see  Cassels \cite{flref4}). We say that two real numbers $\theta$, $\theta'$ are $\pm$~\textsl{equivalent} if their sum or difference is an integer. Notice that $\theta$ is $\pm$~equivalent to $\theta'$ if and only if $\|q\theta\|=\|q\theta'\|$ for every $q\in N$. Hence, if $\theta$ and $\theta'$ are $\pm$~equivalent, then $\varphi(\theta)=\varphi(\theta')$.

The above Markoff result has also the following equivalent formulation in terms of approximation of irrationals by rationals (see also Cassels \cite{flref4}, Cusick--Flahive \cite{flref5}, Hurwitz \cite{flref9} and Schmidt \cite{flref15}):
\\[6pt]
(M1) \quad If $v(\theta) > \frac{1}{3}$, then $\theta$ is equivalent to a root of $f_m(x,1)=0$, where $f_m$ is a Markoff form. Conversely, if $\theta$ is equivalent to a root of $f_m(x,1) =0$, then
$$
v(\theta) = \frac{1}{\sqrt{9-4m^{-2}}} > 1/3\, ,
$$
and there are infinitely many solutions of $q\|q\theta\| < v(\theta)$.
The two roots of $f_m(x,1) = 0$ are equivalent to one another.

The main purpose of this article is to characterize all irrational numbers with constant $\varphi(\theta)>\frac{1}{3}$. We will prove the following theorem (see Theorem~\ref{flthe3.1} and  \ref{flthe4.4}): If $\varphi(\theta) > \frac{1}{3}$, then $\theta$ is $\pm$~equivalent to a root of a quadratic equation {$f_m(x,1) = 0$}, where $f_m$ is a Markoff form. Conversely, if $\theta$ is $\pm$~equivalent to a root of $f_m (x,1) = 0$, then
$$
\varphi(\theta) = m\, \| m\theta \| = \frac{2}{3+\sqrt{9-4m^{-2}}} > 1/3.
$$
Notice that roots both of
$$
f_1(x,1)= x^2 + x - 1 = 0
\quad\mbox{and}\quad
f_2(x,1)= 2x^2 + 4x - 2 = 0
$$
are $\pm$ equivalent. However, for every Markoff number $m > 2$ roots of $f_m(x,1) = 0$ are not $\pm$~equivalent.

A well known conjecture on the uniqueness of Markoff numbers/triples was first mentioned by Frobenius \cite{flref7} (see also D12 in Guy \cite{flref8}). It asserts that a Markoff triple is uniquely determined by its maximal element. The conjecture has only been proved for some special cases. Bargar \cite{flref1}, Button~\cite{flref2} and Schmutz \cite{flref16} proved independently that a Markoff triple $(m,m_1,m_2)$ is unique if  $m$ is either a prime power or twice a prime power. A stronger result has been obtained later by Button in \cite{flref3}; in particular, a Markoff triple $(m,m_1,m_2)$ is unique if $m \leq 10^{35}$. If the uniqueness conjecture is true, then, by Theorems~\ref{flthe3.1} and \ref{flthe4.4}, for every Markoff number $m> 2$, the set of all real numbers $\theta$ such that
$\varphi(\theta) = m \| m\theta \|$
is the union of two classes of $\pm$~equivalence, each of them determined by a root of the quadratic equation $f_m(x,1) = 0$.

\section{The continued fractions of roots of ${f_m(x,1)=0}$}

We use the notation $[a_0,a_1,\ldots]$ for simple continued fraction whose partial quotients are $a_0, a_1, \ldots $ . A subscript attached to a partial quotient will indicate the digit repeated that many times. Also we write a bar over the period. For example,
$[ 0,\overline{2,1_2,2} \,]=[0,2,1,1,2,2,1,1,2, \ldots\,]$.

Frobenius \cite{flref7} associated with each Markoff number $m > 2$ an ordered pair of relatively prime positive integers. These pairs  are called by Cusick--Flahive \cite{flref5} the \textsl{Frobenius coordinates} of Markoff numbers. Let $(m,m_1,m_2)$ be a Markoff triple, $m > 2$, and $(\mu, \nu)$ be the Frobenius coordinates of $m$. Suppose that $f_m(x,y)=mx^2 + (3m-2u)xy + (v-3u)y^2$ is the Markoff form associated to $(m,m_1,m_2)$. By Theorem 3 and 4 in Cusick--Flahive ([5, p.~23 and 27]), the positive root $\alpha_m$ of $f_m(x,1)=0$ has a purely periodic continued fraction with the partial quotients only 1 or 2 satisfying
\begin{equation}\label{fleq2.1}
\alpha_m = [0,2,1_{2r(1)}, 2,2,1_{2r(2)}2,2, \ldots ]
= [0,\overline{2,S(\mu, \nu),1,1,2} \, ]\, ,
\end{equation}
where the symmetric sequence $S(\mu, \nu)$ is given by
\begin{align*}
S(\mu,\nu) &= 1_{2r(1)},2,2,1_{2r(2)}, \ldots , 1_{2r(\nu -1)},2,2,1_{2r(\nu) -2}
\quad \hbox{for} \quad \nu > 1 \, ,
\cr
S(\mu,\nu) & = 1_{2\mu -2}\quad \hbox{for} \quad \nu = 1 \, ,
\end{align*}
and
\[
r(i)= [i\mu/\nu] - [(i-1)\mu/\nu] \quad \hbox{for} \quad i = 1, \ldots , \nu - 1 \, .
\]
Moreover
\begin{equation}\label{fleq2.2}
\frac{v}{u} = [ 0,2, S(\mu,\nu) ]  \quad \hbox{and} \quad
\frac{u}{m} = [ 0,2, S(\mu,\nu),2 ] \, .
\end{equation}
By Lemma~6 in Cusick--Flahive  ([5, p.~30]), the sequence $\{r(i)\}$ is \textsl{balanced}; that is
\[
\Big| \sum\limits^{k+n}_{j=k} r(j) - \sum\limits^{k+n}_{j=k} r(j+s) \Big|
\leq 1 \, ,
\]
for all positive integers $k,s$ and all integers $n\geq 0$.

We shall use the \textsl{Lagrange identity}
\[
[0,2,x] + [0,1,1,x] = 1\, , \quad \hbox{for} \quad x>0 \, .
\]
\begin{lem}\label{fllem2.1}
If $\alpha_m = [0,\overline{2,S(\mu, \nu),1,1,2} \, ]$ is a positive and $\beta_m$ is a negative root of $f_m(x,y) = 0$,  $m>2$, then
\[
-\beta_m -2 = [0,\overline{1,1, S(\mu,\nu), 2,2} \, ] \, ,
\]
and
\[
\beta_m+3 = [0,2, \overline{S(\mu,\nu),2,2,1,1} \, ] \, .
\]
\end{lem}

\begin{proof}
Set $-\gamma-2 = [0, \overline{1,1, S(\mu,\nu),2,2}\, ]$. By the Lagrange identity, we have
$\gamma+3 = [0,2, \overline{S(\mu,\nu),2,2,1,1} \, ]$. We conclude from (\ref{fleq2.2}) that $\sfrac{v}{u}$ and  $\sfrac{u}{m}$ are two consecutive convergents to $\gamma+3$, whence
\[
\gamma+3 = [0,2, S(\mu,\nu),2,-\gamma \, ] = \frac{-\gamma u +v}{-\gamma m +u} \, .
\]
Therefore $m\gamma^2 + (3m-2u)\gamma + v -3u=0$, whence $\gamma = \beta_m$.
%
\end{proof}

\section{The sufficient condition for $\varphi(\theta) > 1/3$}

If $\alpha = [ a_0, a_1 , a_2, \ldots ]$, then the Legendre theorem (see \cite{flref5}, \cite{flref13} or \cite{flref15}) says
\\[6pt]
(L) \quad if $q\|q\alpha\| < \frac{1}{2}$, then $q$ must be equal one of the denominator of the convergents
\[
p_n / q_n = [a_0, a_1, a_2,\ldots , a_n ], \qquad n=0,1,2, \ldots
\]
to $\alpha$.
By the  Perron theorem (see \cite{flref4}, \cite{flref13} or \cite{flref15}), we have
\\[6pt]
(P) $\quad q_n \|q_n \alpha \| = ([0,a_n, a_{n-1}, \ldots , a_1 ]  +
[a_{n+1}, a_{n +2}, a_{n +3}, \ldots \,])^{-1} \, . $
\\[6pt]
The following theorem was proved by Serret (see \cite{flref4}, \cite{flref5} or \cite{flref15}).
\\[6pt]
(S) \quad Suppose $\theta = [a_0,a_1,a_2, \ldots \,]$ and $\theta' = [b_0, b_1, b_2, \ldots \,]$ are
irrational. These numbers are equivalent if and only if there exist integers $k$ and $l$ such that
\[
a_{k+i} = b_{l+i}\, , \quad \hbox{for} \quad i \in N \, .
\]
We will also need the following elementary lemma concerning continued fractions ([5, p.~2] and [6, p.~81]).

\begin{lem}\label{fllem3.1}
Suppose
$$
\alpha = [a_0, a_1, a_2, \ldots , a_n , b_1, b_2 , \ldots \,] \quad \hbox{and} \quad \beta = [a_0, a_1, a_2, \ldots , a_n, c_1, c_2, \ldots \,],
$$
 where $n\geq 0$,
$a_0$ is an integer, and $a_ 1, \ldots , a_n, b_1, b_2,\ldots , c_1, c_2, \ldots \,$ are positive integers with $b_ 1 \neq c_1$.
Then, for $n$ odd, $\alpha < \beta$ if and only if $b_1 < c_1$; for $n$ even,
$\alpha < \beta$ if and only if  \hbox{$b_1 > c_1$}. Also, $\alpha < [a_0, a_1, a_2, \ldots , a_n ]$ when $n$ is odd, and $[a_0, a_1, a_2,\ldots , a_n ] < \alpha $ when $n$ is even.
\end{lem}

\begin{thm}\label{flthe3.1}
If  $\theta$ is $\pm$ equivalent to a root of $f_m(x,1)=0$, where $f_m$ is a~Mar\-koff form, then
\begin{eqnarray}\label{fleq3.1}
m\| m\theta \| & = & \frac{2}{3+\sqrt{9-4m^{-2}} },
\\ \cr
q \| q\theta \| & > &  m\| m \theta \| \quad \hbox{for }
 q\neq m, \ q \in N.\label{fleq3.2}
\end{eqnarray}
\end{thm}

\begin{proof}
It is sufficient to prove (\ref{fleq3.1}) for $\theta = \alpha_m$ and $\theta = \beta_m$, where $\alpha_m$ is the positive and $\beta_m$ is the negative root of $f_m(x,1)=0$. Since
$$
f_m(x,1)=mx^2+ (3m-2u)xy + (v-3u)y^2,
$$
 the discriminant $d(f_m)=9m^2-4$ and $\alpha_ m+\beta_m=\frac{2u}{m} -3$. Hence,
\[
m\beta_m +3m-u  = u-m\alpha_m =
\frac{3m-\sqrt{9m^2-4}}{2}
 = \frac{2}{3m+\sqrt{9m^2 -4}}<\frac{1}{2m} .
\]
Thus, condition (\ref{fleq3.1})  is satisfied for $\theta=\alpha_m$ and $\theta= \beta_m$.

We will prove condition (\ref{fleq3.2}). Suppose first that $m \leq 2$. Then $\alpha_1 = (\sqrt{5} - 1 )/2 = [0,\overline{1}\,]$ and
 $\alpha_2 = (\sqrt{2} - 1 ) = [0,\overline{2}\,]$. Thus $p_n / q_n = [0,1_n]$ ($p_n / q_n = [0,2_n]$) are convergents to $\alpha_1$ ($\alpha_2$, respectively). Lemma~\ref{fllem3.1} yields
 \[
 \begin{array}{l}
 [0,1_n] + [\,\overline{1} \,]  <  [0,1] + [\,\overline{1}\,] \quad \hbox{ for }
  n > 1,
 \cr
  [0,2_n] + [\,\overline{2} \,]  <  [0,2] + [\,\overline{2}\,] \quad \hbox{ for }
  n > 1.
 \end{array}
 \]
 Hence, by (P), $q_n \|q_n\alpha_1\| > \|\alpha_1\|$ ($q_n \|q_n\alpha_2\| >  2 \| 2 \alpha_2\|$, respectively), for $n > 1$. Thus, by (L) and equality (\ref{fleq3.1}), condition (\ref{fleq3.2}) holds for $\theta = \alpha_1$ ($\theta = \alpha_2$). Since $\beta_1$ ($\beta_2$) is $\pm$~equivalent to $\alpha_1$ ($\alpha_2$), condition (\ref{fleq3.2}) is satisfied.

 Now we assume that $m>2$. Our aim is to prove condition (\ref{fleq3.2}) for $\theta = \alpha_m$ and $\theta = -\beta_m -2$. By (\ref{fleq2.1}), we have
 \begin{align*}
 \alpha_m & = [0,2,1_{2r(1)}, 2,2,1_{2r(2)},2,2, \ldots \,]
= [0,\overline{2,S(\mu, \nu),1,1,2} \, ]
\cr
& = [0, \overline{2,1_{2r(1)}, 2,2,1_{2r(2)}, \ldots , 1_{2r(\nu -1)},2,2,1_{2r(\nu)}, 2} \, ],
\end{align*}
 where the sequence $\{r(j)\}$ is balanced. By Lemma~\ref{fllem2.1}, we obtain
 \begin{align*}
 - \beta_m -2 & = [0,1_{2z(1)}, 2,2,1_{2z(2)},2,2, \ldots \,]
= [0,\overline{1,1,S(\mu, \nu),2,2} \, ]
\cr
& = [0,\overline{1_{2z(1)}, 2,2,1_{2z(2)}, \ldots , 1_{2z(\nu -1)},2,2,1_{2z(\nu)}, 2,2} \, ].
\end{align*}
Since $\alpha$ and $\beta$ are equivalent and the sequence $\{r(j)\}$ is balanced, by (S) the purely periodic sequence $\{z(j)\}$ is also balanced. Since the sequence $S(\mu,\nu)$ is symmetric, we have
 \[
 \begin{array}{rllrll}
r(1) & = & r(\nu)-1, \quad & z(1) &=& z(v)+1,
\cr
r(j+1) & = & r(\nu-j), \quad & z(j+1) &=& z(\nu -j), \hbox{ for }  1 \leq j \leq \nu -2.
\end{array}
\]
 Hence,
\begin{equation}\label{fleq3.3}
\begin{array}{lcl}
     \sum\limits^k_{j=1} r(j)& < & \sum\limits^{k-1}_{j=0}
 r(\nu -j),
 \cr
  \sum\limits^k_{j=1} z(j) & > & \sum\limits^{k-1}_{j=0} z(\nu-j),   \hbox{ for } 1\leq k < \nu.
\end{array}
\end{equation}
 Since the sequences $\{r(j)\}$ and $\{z(j)\}$ are purely periodic with the period~$\nu$, we obtain
\begin{equation}\label{fleq3.4}
\begin{array}{lcl}
 \sum\limits^{l\nu}_{j=1} r(s+j) & = &  \sum\limits^{l\nu}_{j=1} r(j),
 \cr
  \sum\limits^{l\nu}_{j=1} z(s +j) & = & \sum\limits^{l\nu}_{j=1} z(j),
 \hbox{ for $s$, $l \in N$}.
\end{array}
\end{equation}
Let us denote
\[
\alpha_m  = [0,a_1, a_2, \ldots \,] = [0,\overline{2,S(\mu, \nu),1,1,2} \, ]
\]
and
\[
-\beta_m -2  = [0,b_1,b_2, \ldots \,] = [0,\overline{1,1,S(\mu, \nu),2,2} \, ].
\]
By condition (\ref{fleq2.1}) and the Lagrange identity we have
\[
\frac{u}{m} = [ 0,2, S(\mu,\nu),1,1 ]
\quad\mbox{and}\quad
1- \frac{u}{m}= [0,1,1,S(\mu,\nu),2 ].
\]
Hence, $\frac{u}{m}$ and $1- \frac{u}{m}$ are convergents to $\alpha_m$ and $-\beta_m -2$, respectively. Thus, by (P), (L) and equality (\ref{fleq3.1}), it is sufficient to prove that
\begin{align*}
[0,a_n, a_{n-1},& \ldots , a_1] + a_{n+1} + [0, a_{n+2}, a_{n+3}, \ldots \,]
\\& < [0,1,1,S(\mu,\nu), 2] + 2 +
[0,\overline{2,S(\mu,\nu),1,1,2} \,]
\end{align*}
for $[0,a_1, \ldots , a_n ] \neq [0,2,S(\mu,\nu), 1,1 ]$, and
\begin{align*}
[0,b_n, b_{n-1},& \ldots , b_1] + b_{n+1} + [0, b_{n+2}, b_{n+3}, \ldots \, ]
\\& <[0,2,S(\mu,\nu),1, 1] + 2 + [0,\overline{1,1,S(\mu,\nu),2,2} \,]
\end{align*}
for  $[0,b_1, \ldots , b_n ]\neq [0,1,1,S(\mu,\nu),2 ]$.
We assume that
\[
[0,a_1, \ldots , a_n ] \neq [0,2,S(\mu,\nu), 1,1 ]
\quad
\mbox{and}
\quad
 [0,b_1, \ldots , b_n ] \neq
[0,1,1,S(\mu,\nu),2 ].
\]
Let us consider the following cases:
\begin{flushleft}
\begin{tabular}{l}
$\hbox{(1)} \qquad a_{n+1} = 2  \quad (\hbox{or }   b_{n+1} =2)
  \quad \hbox{and} \ n \ \hbox{is odd},$
  \\
   $\hbox{(2)} \qquad a_{n+1} = 2 \quad (\hbox{or }   b_{n+1} =2)
  \quad \hbox{and} \ n \ \hbox{is even},$
  \\
   $\hbox{(3)} \qquad a_{n+1} = 1 \quad (\hbox{or }  b_{n+1} =1)
  \quad \hbox{and} \ n \ \hbox{is odd},$
  \\
   $\hbox{(4)} \qquad a_{n+1} = 1  \quad (\hbox{or }   b_{n+1} =1)
  \quad \hbox{and} \ n \ \hbox{is even}.$
  \end{tabular}
  \end{flushleft}
Case (1). There exists $s\geq1$, $s\neq \nu$, such that
 \[
   [0,a_1, \ldots , a_n] = [0,2,1_{2r(1)}, 2,2,1_{2r(2)}, \ldots , 1_{2r(s)}]
 \]
  ({or}
  \[
   [0,b_1, \ldots , b_n] =
   [0,1_{2z(1)}, 2,2,1_{2z(2)}, \ldots , 1_{2z(s)},2]),
    \]
 respectively).

\noindent
 Since the sequence $\{r(j)\}$ (or $\{z(j)\}$) is balanced, by condition (\ref{fleq3.3}) we have
 \[
 \begin{array}{ll}
 &\sum\limits^{k-1}_{j=0} r(s-j) \leq \sum\limits^{k-1}_{j=0} r(\nu -j),
 \hbox{ for }  1\leq k \leq \min(\nu,s),
 \cr
 &\sum\limits^{s-1}_{j=0} r(s-j)  < \sum\limits^{s-1}_{j=0} r(\nu -j),
 \hbox{ for }  s < \nu,
 \cr
 (\hbox{or} \quad
 &\sum\limits^{k-1}_{j=0} z(s-j) \geq \sum\limits^{k-1}_{j=0} z(\nu -j),
 \hbox{ for }  1\leq k \leq \min(\nu,s),
 \cr
 &\sum\limits^{s-1}_{j=0} z(s-j) > \sum\limits^{s-1}_{j=0} z(\nu -j),
 \hbox{ for }  s < \nu).
 \end{array}
 \]
 Hence, if $s<\nu$, then by Lemma~\ref{fllem3.1}, it follows that
\[
[0,1_{2r(s)}, 2,2, \ldots , 1_{2r(1)},2 ]
 <  [0,1_{2r(\nu)}, 2,2, \ldots , 1_{2r(1)},2] = [0,1,1,S(\mu,\nu),2]
\]
({or}
\[
[0,2,1_{2z(s)},2,2, \ldots , 1_{2z(1)}] < [0,2,1_{2z(\nu)},2,2, \ldots , 1_{2z(1}] = [0,2,S(\mu,\nu),1,1 ] ,
\]
respectively).

\noindent
Indeed, after pairing off the consecutive equal terms 1 or 2 on the left and on the right, we must come to a term 2 on the left (on the right, respectively) and the corresponding term 1 on the right (on the left, respectively).
If $s>\nu$, then by Lemma~\ref{fllem3.1}, it follows that
\begin{align*}
[0,1_{2r(s)}, 2,2,& \ldots , 1_{2r(1)},2 ]
 < [0,1_{2r(s)}, 2,2, \ldots , 1_{2r(s-\nu + 1)},2 ]\\
 &\leq  [0,1_{2r(\nu)}, 2,2, \ldots , 1_{2r(1)},2] = [0,1,1,S(\mu,\nu),2]
\end{align*}
({or}
\begin{align*}
  [0,2,1_{2z(s)},2,2,& \ldots , 1_{2z(1)} ]
 <  [0,2,1_{2z(s)}, 2,2, \ldots , 1_{2z(s-\nu + 1)}]\\
 & \leq [0,2,1_{2z(v)}, 2,2, \ldots , 1_{2z(1)}]=[0,2,S(\mu,\nu),1,1 ],
 \end{align*}
respectively).

\noindent
Analogically, by conditions (\ref{fleq3.3}) and (\ref{fleq3.4}) we have
 \[
 \begin{array}{ll}
 &\sum\limits^{k}_{j=1} r(s+j)  \geq  \sum\limits^{k}_{j=1} r(j),
 \hbox{ for $k \in N$} ,
 \cr
 (\hbox{or} \quad
  &\sum\limits^{k}_{j=1} z(s+j)  \leq  \sum\limits^{k}_{j=1} z(j),
 \hbox{ for $k \in N$)}.
 \end{array}
 \]
Hence, by Lemma~\ref{fllem3.1}, we obtain
\begin{align*}
[0,2,1_{2r(s+1)},2,2, 1_{2r(s+2)}, \ldots \,] & \leq [0,2,1_{2r(1)}, 2,2,1_{2r(2)}, \ldots \, ] \\ & =  [0,\overline{2,S(\mu,\nu),1,1,2} \, ]
\end{align*}
 ({or}
\begin{align*}
 [0,1_{2z(s+1)},2,2,1_{2z(s+2)}, \ldots \,] & \leq [0,1_{2z(1)}, 2,2,1_{2z(2)}, \ldots \,] \\&  = [0,\overline{1,1,S(\mu,\nu), 2,2}\,],
 \end{align*}
  respectively)\,.
\begin{flushleft}
Case (2). There exists $s \geq 1$ such that
\end{flushleft}
\[
[0,a_1, \ldots , a_n] =[0,2,1_{2r(1)},2,2, 1_{2r(2)}, \ldots , 1_{2r(s)},2 ]
\]
({or}
\[
 [0,b_1, \ldots , b_n ] =
 [0,1_{2z(1)},2,2,1_{2z(2)}, \ldots , 1_{2z(s)}  ],
\]
respectively).

\noindent
Since the sequence $\{r(j)\}$ (or $\{z(j)\}$) is balanced, by conditions  (\ref{fleq3.3}) and (\ref{fleq3.4}), we have
\[
 \begin{array}{lll}
 &\sum\limits^{k-1}_{j=0} r(s-j) \geq \sum\limits^{k}_{j=1} r(j), &
 \hbox{ for} \quad  1 \leq k \leq s,
 \cr
 &\sum\limits^{k}_{j=1} r(s+j) \leq \sum\limits^{k-1}_{j=0} r(\nu -j), &
 \hbox{ for} \quad  1 \leq k \leq  \nu,
 \cr
 (\hbox{or} \quad
 &\sum\limits^{k-1}_{j=0} z(s-j) \leq \sum\limits^{k}_{j=1} z(j), &
 \hbox{ for} \quad  1 \leq k \leq s,
 \cr
 &\sum\limits^{k}_{j=1} z(s+j) \geq \sum\limits^{k-1}_{j=0} z(\nu -j), &
 \hbox{ for} \quad  1 \leq k \leq  \nu).
 \end{array}
 \]
Hence, by Lemma~\ref{fllem3.1}, we have
\begin{align*}
  [0,2,1_{2r(s)},2,2, \ldots, 1_{2r(1)},2]& \leq [0,2,1_{2r(1)},2,2, \ldots, 1_{2r(s)},2]\\ & <
  [0,\overline{2,S(\mu,\nu),1,1,2} \, ]
\end{align*}
 {and}
\begin{align*}
 [0,1_{2r(s+1)},2,2,1_{2r(s+2)}, \ldots \,] & < [0,1_{2r(\nu)}, 2,2, \ldots , 1_{2r(1)},2] \\ & =
 [0,1,1, S(\mu,\nu), 2]
 \end{align*}
 ({or}
 \begin{align*}
  [0,1_{2z(s)},2,2, \ldots, 1_{2z(1)}]& \leq [0,1_{2z(1)},2,2, \ldots, 1_{2z(s)}]\\
 & < [0,\overline{1,1, S(\mu,\nu), 2,2}\,]
  \end{align*}
  {and}
  \begin{align*}
   [0,2,1_{2z(s+1)},2,2, 1_{2z(s+2)}, \ldots \, ] & < [0,2,1_{2z(\nu)}, 2,2, \ldots , 1_{2z(1)}]\\ & =
    [0,2, S(\mu,\nu),1,1],
  \end{align*}
  respectively).
\begin{flushleft}
Case (3). There exists $s \geq 1$ and $0< k \leq r(s)$ (or $0 \leq k < z(s)$)
such that
\end{flushleft}
\[
[0,a_1, \ldots , a_n] =[0,2,1_{2r(1)},2,2, 1_{2r(2)}, \ldots , 1_{2r(s)-2k} ]
\]
({or}
\[
 [0,b_1, \ldots , b_n ] =
 [0,1_{2z(1)},2,2,1_{2z(2)}, \ldots , 1_{2z(s) -2k-1}  ],
\]
respectively).

\noindent
Since $r(s)-k < r(s) \leq r(\nu)$ (or $k < z(s) \leq z(1)$), Lemma~\ref{fllem3.1} yields
\begin{align*}
[0,1_{2r(s)-2k},2,2, 1_{2r(s-1)}, \ldots , 1_{2r(1)} ,2 ] & < [0,1_{2r(\nu)}, 2,2, \ldots , 1_{2r(1)},2 \,] \\ & =
 [0,1,1, S(\mu,\nu), 2]
 \end{align*}
({or}
\begin{align*}
 [0,1_{2k},2,2,1_{2z(s+1)}, \ldots \, ] & < [0,1_{2z(1)}, 2,2,1_{2z(2)}, \ldots \,]
 \\ & = [0, \overline{1,1, S(\mu,\nu),2,2}\,],
 \end{align*}
respectively)\,.
\begin{flushleft}
Case (4). There exists $s \geq 1$ and $0\leq k < r(s)$ (or $0 < k \leq z(s)$)
such that
\end{flushleft}
\[
[0,a_1, \ldots , a_n] =[0,2,1_{2r(1)},2,2, 1_{2r(2)}, \ldots , 1_{2r(s)-2k-1} ]
\]
({or}
\[
 [0,b_1, \ldots , b_n ] =
 [0,1_{2z(1)},2,2,1_{2z(2)}, \ldots , 1_{2z(s) -2k} ],
 \]
respectively).

\noindent
Since $k < r(s) \leq r(\nu)$ (or $ z(s)-k < z(s) \leq z(1)$), Lemma~\ref{fllem3.1} yields
\begin{align*}
[0,1_{2k},2,2, 1_{2r(s+1)}, \ldots \,] & < [0,1_{2r(\nu)}, 2,2, \ldots , 1_{2r(1)},2 \,] \\ & =
 [0,1,1, S(\mu,\nu), 2]
\end{align*}
({or}
\begin{align*}
 [0,1_{2z(s)-2k},2,2,1_{2z(s-1)}, \ldots  , 1_{2z(1)}, 2 ] & < [0,1_{2z(1)}, 2,2,1_{2z(2)}, \ldots \,] \\
 & = [0, \overline{1,1, S(\mu,\nu),2,2}\,],
\end{align*}
respectively).
\end{proof}

\section{The necessary condition for ${\varphi({\theta})  > 1/3}$}
Markoff \cite{flref11} (see Cusick--Flahive \cite{flref5} and Dickson \cite{flref6}) characterized all doubly infinite sequences of positive integers $A_d = \ldots , a_{-1}, a_{-2}, a_0, a_1, a_2 , \ldots $ such that
\[
\lambda_n (A_d) = [0,a_n, a_{n-1}, \ldots ] +
[a_{n+1}, a_{n+2}, \ldots ] < 3 \quad \hbox{for }  n \in N.
\]
In particular,
\\[6pt]
(M2) if $\lambda_n(A_d) < 3$  for all $ n$, and $A_d \neq \overline{1}$, then $A_d$ is of the form
\[
\ldots,2,2,1_{2r(-1)},2,2,1_{2r(0)},2,2,1_{2r(1)},2,2, \ldots
\]
for some sequence $r(i)$ of nonnegative integers which satisfies
\[
|r(i+1)-r(i)| \leq 1 \quad \hbox{for all }  i.
\]
By analogy to the Markoff result (see Theorem~3 in Cusick--Flahive [5, p.~6]), in Theorem~\ref{flthe4.1} we characterize all sequences of positive integers $A= a_1, a_2 , \ldots $ such that $a_1 \geq 2$ and
\[
\mu_n (A) = [0,a_n,a_{n-1}, \ldots , a_1] +
[ a_{n+1}, a_{n+2} , \ldots ] < 3\quad \hbox{for } n \in N.
\]
We will need the following lemma which follows immediately from the Lagrange identity and Lemma~\ref{fllem3.1} (see Lemma~3 in Cusick--Flahive [5,~p.~3]).
\begin{lem}\label{fllem4.1}
For each even integer $n \geq 2$ and any real $x\geq 1$, $y\geq 1$, we have the equivalence:
$[2,1_n,x] + [0,2,1_{n-2},y] \leq 3$ if and only if $x \geq y$, where on the left the equality holds if and only if $x = y$.
\end{lem}
\begin{thm}\label{flthe4.1} Let  $A= a_1, a_2 , \ldots $ be a sequence of positive integers such that $a_1 \geq 2$ and $A\neq 2, \overline{1}$.
In order to have $\mu_n(A) <3$  for all $ n$ it is necessary and sufficient that $A$ has the form
\begin{equation}\label{fleq4.1}
2,1_{2r(1)},2,2,1_{2r(2)},2,2,1_{2r(3)},2,2, \ldots
\end{equation}
where the $r(i)$ are nonnegative integers with the properties:
\begin{description}
\item[\mbox{\rm(A)}\ \ ]  $|r(i+1)-r(i)| \leq 1$, for all $i$;
\item[\mbox{\rm(B)}\ \ ] if $r(i+1)-r(i)$ is $-1$ or $+1$, respectively, then the first of
 the integers $r(i+j+1)-r(i-j)$ ($ 1 \leq j < i$) which is not zero is positive or negative, respectively;
\item[\mbox{$(\mathrm{C}_{01})$} \ ] $r(1) \leq r(2)$. Moreover, if $ r(i+1)- r(i)= -1$, then there exists
 $ 1 \leq j < i$ such that $r(i+j+1) > r(i-j)$.
\end{description}
\end{thm}
\begin{proof}
The proof that $A$ has the form (\ref{fleq4.1}) and satisfies condition (A) is similar to that in (M2) (see Cusick--Flahive [5, Theorem~2 and Corollary, p.~5]), and will be omitted. Let $A$ be of the form (\ref{fleq4.1}). Set
\[
\begin{array}{ll}
x(i):= &
\left\{
\begin{array}{ll}
2 & \quad \hbox{for }  i=1,
\\
{}[2,2,1_{2r(i-1)}, \ldots ,1_{2r(1)},2] & \quad \hbox{for } i>1,
\end{array}
\right.
\\[12pt]
y(i) := & \ [2,2,1_{2r(i+2)} , \ldots ].
\end{array}
\]
Suppose, in contradiction to (B), that there exists a pair of integers $k < l$ such that $r(l+1)= r(l)+1$, $r(l+j+1)= r(l-j)$, for $1\leq j < k$, and $r(l+k+1)> r(l-k)$. By Lemma~\ref{fllem3.1},
$y(l) < x(l)$. Thus Lemma~\ref{fllem4.1} gives
\[
\mu_l (A) = [2,1_{2r(l+1)},y(l)] + [0,2,1 _{2r(l)},x(l) ] > 3.
\]
Now assume, in contradiction to (C$_{01}$), that there exists integer $l$ such that $r(l+1)=r(l)-1$ and $r(l+j+1)= r(l-j)$, for $1\leq j < l$. By Lemma~\ref{fllem3.1},
$x(l) < y(l)$. Thus, Lemma~\ref{fllem4.1} gives
\[
\mu_l (A) = [2,1_{2r(l)},x(l)] + [0,2,1 _{2r(l+1)},y(l) ] > 3.
\]

For the sufficiency part of the theorem consider any $i\geq 1$ and let ${x = x(i)}$ and $y=y(i)$. By (A), we have $|r(i)-r(i+1)|\leq 1$. If $2r(i) = n+2$ and $2r(i+1)= n$, then $x > y > [1_2, x]$, because of (B), (C$_{01}$) and Lemma~\ref{fllem3.1}. Hence, Lemma~\ref{fllem4.1} gives
\[
\begin{array}{l}
[2,1_{n+2},x] + [0,2,1_{n},y] < 3,
\cr
[2,1_{n}, y] + [0,2,1_{n-2}, 1_{2},x ] < 3.
\end{array}
\]
If $2r(i)= n$ and $2r(i+1)= n+2$, then $y > x > [1_2,y]$, because of (B) and Lemma~\ref{fllem3.1}. Hence, Lemma~\ref{fllem4.1} gives
\[
\begin{array}{l}
[2,1_{n+2},y] + [0,2,1_{n},x] < 3,
\cr
[2,1_{n},x ] + [0,2,1_{n-2},1_{2}, y] < 3.
\end{array}
\]
If $2r(i)= 2r(i+1)= n$, then Lemma~\ref{fllem4.1} immediately gives
\[
\begin{array}{l}
[2,1_{n},x ] + [0,2,1_{n-2},1_{2}, y] < 3,
\cr
[2,1_{n},y] + [0,2,1_{n-2},1_{2},x] < 3.
\end{array}
\]
Combining these results, we get $\mu_n(A) < 3$ for all $n$.
\end{proof}

 A sequence $\{r(i)\}$  of nonnegative integers is said to be \textsl{Markoff balanced} (see Cusick--Flahive [5, p.~28]) if it satisfies the conditions (A) and (B) of Theorem~\ref{flthe4.1}. Let ${\mathcal M}_{01}$ be the family of all Markoff balanced sequences $\{r(i)\}$ which  satisfy the condition (C$_{01}$) of Theorem~\ref{flthe4.1}, and suppose that ${\mathcal M}_{10}$ is the family of all Markoff balanced sequences $\{t(i)\}$ which  satisfy the condition
 \begin{description}
  \item[\mbox{$(\mathrm{C}_{10})$} \ ] $t(1) \geq t(2)$. Moreover, if $t(i+1) - t(i) = +1$, then there exists
  $1 \leq j< i$ such that $t(i+j+1)< t(i-j)$.
\end{description}
By analogy to the Markoff result (see Theorem~4 in Cusick--Flahive [5, p.~7]) in Theorem~\ref{flthe4.2} and Corollary~\ref{corollary4.1} we characterize all Markoff balanced sequences of the family ${\mathcal M}_{01}$ and ${\mathcal M}_{10}$, respectively.

We say that two sequences $\{r(i)\}$ and $\{r'(i)\}$ are \textsl{equivalent} if and only if there exist integers $k$ and $l$ such that
\[
r(k+i) = r'(l+i), \quad \hbox{for } i \in N.
\]
\begin{thm}\label{flthe4.2}
Suppose a sequence $\{r(i)\}\in {\mathcal M}_{01}$, $\{r(i)\} \neq \overline{0}$. Then there is a positive integer $r$ such that the sequence $\{r(i)\}$ is of one of the types
\begin{description}
 \item[$R_{0}$] $\qquad  \overline{r}$,
\item[$R_{01}$] $ \qquad  r-1,\overline{r}$,
\item[$R$ ]  \qquad $ (r-1)_{s(1)}, r, (r-1)_{s(2)}, r,(r-1)_{s(3)},r, \ldots$,
\end{description}
where the sequence $\{s(i)\} \in {\mathcal M}_{10}$, and $s(i) \neq 0$ for infinitely many $i$.
\end{thm}
\begin{proof}
Let $\{r(i)\} \neq \overline{0}$ be a sequence in ${\mathcal M}_{01}$. We begin by proving the following universal inequality
\\[6pt]
(i) \quad
$|r(i)-r(j)| \leq 1$ \quad for all $i$ and $j$.
\\[6pt]
Suppose, on the contrary, that $|r(i)-r(j)| \geq 2$ for some $i < j$. Hence, by (A) and (B), there is a minimal positive integer $k$ such that $|r(i)- r(i+k+1)| = 2$, for some $i$. We assume that $r(i)= n$ and $r(i+k+1) = n+2$ (a similar argument takes care of the $r(i)= n+2$ and $r(i+k+1) = n$ case). If $k = 1$, then the pattern
\[
r(i)= n,\ r(i+1)= n+1,\ r(i+2)= n+2
\]
occurs in $\{r(i)\}$. Then we would have $r(i+3) \geq n+1$ by (A), which contradicts $r(i+3) \leq n$ by (B). Hence $k \geq 2$, and the pattern
\[
{r(i)} = n,\ {r(i+1)} = \ldots = {r(i+k)} = n+1,\ {r(i+k+1)} = n+2
\]
must appear in $\{r(i)\}$. Then, by (A) and (B),
$$
{r(i+k+2)} = \ldots = {r(i+2k)} = n+1,\
{r(i+2k+1)} = n,
$$
which contradicts the minimality of $k$. So we see that (i) holds.

If $\{r(i)\}$ is equivalent to a constant sequence, then it is of the type  $\overline{r}$ or  $r-1,\overline{r}$, for some positive integer $r$. Otherwise, by (i) and (B), $\{r(i)\}$ is of the type $(r-1)_m, r, \overline{r-1}$ with $m \geq 0$, or $r_n, r-1, \overline{r}$ with $n \geq 1$, which is false by (C$_{01}$).

If $\{r(i)\}$ is not equivalent to a constant sequence, then, by (i), there is a positive integer $r$ such that $\{r(i)\}$ is of the type
$$ (r-1)_{s(1)}, r, (r-1)_{s(2)}, r,(r-1)_{s(3)},r, \ldots ,$$
where $\{s(i)\}$ is a sequence of nonnegative integers and it is not equivalent to $\overline{0}$. We proceed to show that the sequence $\{s(i)\} \in {\mathcal M}_{10}$. Indeed, let $r(k)$ denote the $r$ immediately following $(r-1)_{s(i)}$. If $s(i) >  s(i+1)+ 1$, then $r(k) - r(k-1) = 1$ and $r(k+j) - r(k-j-1) = 0$, for $j = 1, 2, \ldots, s(i+1)$, and is equal to +1 for $j = s(i+1) +1$. This contradicts (B). Therefore, $s(i) \leq s(i+1)+1$, and a~similar argument shows $s(i+1) \leq s(i)+1$. Hence, we have proved that
\\[6pt]
(ii) \quad $|s(i+1)-s(i)| \leq 1$ \quad for all $i$.
\\[6pt]
The inequality $s(1)~ \geq~ s(2)$ follows from (C$_{01}$). Let $s(i+1)-s(i) = 1$. By (B) and (C$_{01}$), after pairing off the equal terms $r-1$ or $r$ equidistant from the central pair $r(k)$, $r(k+1)$, we must come to a term $r$ on the right and an equidistant term $r-1$ on the left. Thus the latter term is one of a block of consecutive terms $r-1$ which is longer then the corresponding block of consecutive terms $r-1$ on the right. Hence, the first nonzero difference $s(i+j+i) - s(i-j)$ is negative. So we see that
\\[6pt]
(iii) \quad  $s(1) \geq s(2)$. Moreover,
if $s(i+1)-s(i) = + 1$, then the first of the integers $s(i+j+1)-s(i-j)$ ($ 1 \leq j < i$) which is not zero is negative, and such negative integer does actually occur.
\\[6pt]
In a similar way we can see that
\\[6pt]
(iv) \quad
if $s(i+1)-s(i) = - 1$, then the first of the integers $s(i+j+1)-s(i-j)$ ($ 1 \leq j < i$) which is not zero is positive.
\end{proof}
Likewise, we can prove the following corollary.
\begin{cor}\label{corollary4.1}
Suppose a sequence $\{t(i)\} \in {\mathcal M}_{10}$.
Then there is a nonnegative integer $t$ such that the sequence $\{t(i)\}$ is of one of the types
\begin{description}
\item[$T_{0}$] $ \qquad  \overline{t}$,
\item[$T_{10}$] $\qquad   t+1,\overline{t}$,
\item[$T$ ]  $\qquad   (t+1)_{u(1)}, t, (t+1)_{u(2)}, t, (t+1)_{u(3)}, t, \ldots,$
\end{description}
where the sequence  $\{u(i)\} \in {\mathcal M}_{01}$, and $u(i) \neq 0$ for infinitely many $i$.
\end{cor}
\begin{thm}\label{flthe4.3}
For every periodic (not necessarily purely periodic) sequence $\{r(i)\}$ in ${\mathcal M}_{01}$, $\{r(i)\} \neq \overline{0}$, there exists exactly one more sequence in ${\mathcal M}_{01}$ which is equivalent to  $\{r(i)\}$. For every periodic sequence $\{t(i)\}$ in ${\mathcal M}_{10}$ there exists exactly one more sequence in ${\mathcal M}_{10}$, which is equivalent to  $\{t(i)\}$.
\end{thm}
\begin{proof}
The proof proceeds by induction on the period of a sequence. If $\{r(i)\}$ ($\{t(i)\}$) has a period equal to $1$, then by Theorem~\ref{flthe4.2} (Corollary~\ref{corollary4.1}), it is of type {$R_{0}$} or {$R_{01}$}  ({$T_{0}$} or {$T_{10}$}, respectively) and the theorem holds. If $\{r(i)\}$ ($\{t(i)\}$) is of type {R} ({T}), then, by Theorem~\ref{flthe4.2} (Corollary~\ref{corollary4.1}), it is determined by a periodic sequence $\{s(i)\} \in {\mathcal M}_{10}$ ($\{u(i)\} \in {\mathcal M}_{01})$ with a period smaller than that of $\{r(i)\}$ ($\{t(i)\}$, respectively). Since there exists just one more sequence in ${\mathcal M}_{10}$ (${\mathcal M}_{01}$) which is equivalent to  $\{s(i)\}$ ($\{u(i)\}$), we have just one more sequence in ${\mathcal M}_{01}$ (${\mathcal M}_{10}$) which is equivalent to  $\{r(i)\}$ ($\{t(i)\}$, respectively).
\end{proof}
Let $\theta = [0,a_1,a_2, \ldots]$ and  $A= a_1, a_2 , \ldots $. By the Legendre Theorem (L)
and the Perron Theorem (P) we obtain the following formula
\begin{equation}\label{fleq4.2}
1/\varphi(\theta) = \sup \{ \mu _n (A): n \in N \}.
\end{equation}
 \begin{thm}\label{flthe4.4}
 If $\varphi(\theta) >\frac{1}{3}$, then $\theta$ is $\pm$~equivalent to a root of \hbox{$f_m(x,1)=0$}, where $f_m$ is a~Markoff form.
 \end{thm}
\begin{proof}
Let $\alpha_m$ and $\beta_m$ be the positive and the negative root of the quadratic equation $f_m(x,1)=0$, where $f_ m$ is a Markoff form.  By Theorem~\ref{flthe3.1} and Lemma~\ref{fllem2.1}, we have
\\[6pt]
(i) \quad
$\varphi(\alpha_m) >\frac{1}{3}$ and $\varphi(\beta_m + 3) >\frac{1}{3}$,
\\[6pt]
(ii)\quad $0 <\alpha_m < \beta_m  + 3 < \frac{1}{2}$, for $m > 2$.
\\[6pt]
Let $\varphi(\theta) > \frac{1}{3}$. Since $\nu(\theta) \geq \varphi(\theta) > \frac{1}{3}$, by the Markoff theorem  (M1) we obtain
 \\[6pt]
(iii) \quad $\theta$, $\alpha_m$ and $\beta_m  + 3$ are equivalent, for some Markoff number $m$.
\\[6pt]
Since every real number is $\pm$~equivalent to an element of the interval $(0,\frac{1}{2})$, we can assume that $\theta =[0,a_1,a_2, \ldots]$ and $a_1 \geq 2$. Suppose that $\theta \neq [0,2,\overline{1}]$ and $\theta \neq [0,\overline{2}]$. By condition (\ref{fleq4.2}) and Theorem~\ref{flthe4.1}, $\{a_i\}$ is of the form
\[
2,1_{2r(1)},2,2,1_{2r(2)},2,2,1_{2r(3)},2,2, \ldots
\]
where $\{r(i)\} \in {\mathcal M}_{01}$. Since $\theta \neq [0,\overline{2}]$, Theorem \ref{flthe4.2} shows that $\{r(i)\}$ is not equivalent to $\overline{0}$. Hence, we have
 \\[6pt]
(iv) $\theta$ is equivalent neither to $\alpha_1 = [0,\overline{1}]$ nor to $\alpha_2 = [0,\overline{2}]$.
 \\[6pt]
By (iii), the sequence $\{r(i)\}$ is periodic. Hence, by Theorem~\ref{flthe4.3} there exists just one more sequence in ${\mathcal M}_{01}$ which is equivalent to  $\{r(i)\}$. Accordingly, by condition (\ref{fleq4.2}) and Theorem~\ref{flthe4.1}
there exists just one more real number $\theta' \in (0,\frac{1}{2})$ with $\varphi(\theta') >\frac{1}{3}$, which is equivalent to  $\theta$.
Hence, from (i)--(iv) we conclude that $\theta = \alpha_m$ or $\theta = \beta_m  + 3$, for some $m > 2$.
\end{proof}

\end{document}